\documentclass[12pt]{article}
\usepackage{amsmath, amsthm, amsfonts}
\usepackage{amssymb}
\usepackage{amscd}
\usepackage{verbatim}
\begin{document}
\newcommand{\eqdef}{\stackrel{{\mathrm{def}}}{=}}
\newcommand{\coreplus}{C_{0+}^{\infty}(\Omega)}
\newcommand{\loc}{{\mathrm{loc}}}
\newcommand{\dx}{\,\mathrm{d}x}
\newcommand{\dy}{\,\mathrm{d}y}
\newcommand{\dz}{\,\mathrm{d}z}
\newcommand{\dt}{\,\mathrm{d}t}
\newcommand{\du}{\,\mathrm{d}u}
\newcommand{\dv}{\,\mathrm{d}v}
\newcommand{\dV}{\,\mathrm{d}V}
\newcommand{\ds}{\,\mathrm{d}s}
\newcommand{\dr}{\,\mathrm{d}r}
\newcommand{\dS}{\,\mathrm{d}S}
\newcommand{\drho}{\,\mathrm{d}\rho}
\newcommand{\core}{C_0^{\infty}(\Omega)}
\newcommand{\sob}{W^{1,p}(\Omega)}
\newcommand{\sobloc}{W^{1,p}_{\mathrm{loc}}(\Omega)}
\newcommand{\merhav}{{\mathcal D}^{1,p}}
\newcommand{\be}{\begin{equation}}
\newcommand{\ee}{\end{equation}}
\newcommand{\mysection}[1]{\section{#1}\setcounter{equation}{0}}
%%%%%%%%%%%%%%%
\newcommand{\bea}{\begin{eqnarray}}
\newcommand{\eea}{\end{eqnarray}}
\newcommand{\bean}{\begin{eqnarray*}}
\newcommand{\eean}{\end{eqnarray*}}
\newcommand{\thkl}{\rule[-.5mm]{.3mm}{3mm}}
%%%%%%%%%%%%%%%%%%%%%%%%%%%
\newcommand{\cw}{\stackrel{\rightharpoonup}{\rightharpoonup}}
\newcommand{\id}{\operatorname{id}}
\newcommand{\supp}{\operatorname{supp}}
\newcommand{\wlim}{\mbox{ w-lim }}
\newcommand{\mymu}{{x_N^{-p_*}}}
\newcommand{\R}{{\mathbb R}}
\newcommand{\N}{{\mathbb N}}
\newcommand{\Z}{{\mathbb Z}}
\newcommand{\Q}{{\mathbb Q}}
\newcommand{\abs}[1]{\lvert#1\rvert}
%%%%%%%%%%%
\newtheorem{theorem}{Theorem}[section]
\newtheorem{corollary}[theorem]{Corollary}
\newtheorem{lemma}[theorem]{Lemma}
\newtheorem{definition}[theorem]{Definition}
\newtheorem{remark}[theorem]{Remark}
\newtheorem{remarks}{Remarks}[section]
\newtheorem{proposition}[theorem]{Proposition}
\newtheorem{problem}[theorem]{Problem}
%%%%%%%%%%%%%%%%%%
\newtheorem{conjecture}[theorem]{Conjecture}
\newtheorem{question}[theorem]{Question}
\newtheorem{example}[theorem]{Example}
\newtheorem{Thm}[theorem]{Theorem}
\newtheorem{Lem}[theorem]{Lemma}
\newtheorem{Pro}[theorem]{Proposition}
\newtheorem{Def}[theorem]{Definition}
\newtheorem{Exa}[theorem]{Example}
\newtheorem{Exs}[theorem]{Examples}
\newtheorem{Rems}[theorem]{Remarks}
\newtheorem{Rem}[theorem]{Remark}
\newtheorem{Cor}[theorem]{Corollary}
\newtheorem{Conj}[theorem]{Conjecture}
\newtheorem{Prob}[theorem]{Problem}
\newtheorem{Ques}[theorem]{Question}
\newcommand{\pf}{\noindent \mbox{{\bf Proof}: }}
%%%%%%%%%%%%%%%%%%
%\newenvironment{proof}{{\bf Proof.}}{\hfill $\bowtie$\vskip4mm}
\renewcommand{\theequation}{\thesection.\arabic{equation}}
\catcode`@=11 \@addtoreset{equation}{section} \catcode`@=12
%\begin{titlepage}
%%%%%%%%%%%%%%%%%%
\title{On the Hardy-Sobolev-Maz'ya inequality\\ and its
generalizations}
\author{Yehuda Pinchover\\
 {\small Department of Mathematics}\\ {\small  Technion - Israel Institute of Technology}\\
 {\small Haifa 32000, Israel}\\
{\small pincho@techunix.technion.ac.il}\\\and Kyril Tintarev
\\{\small Department of Mathematics}\\{\small Uppsala University}\\
{\small SE-751 06 Uppsala, Sweden}\\{\small
kyril.tintarev@math.uu.se}}
%\date{\small{Received ---------------} \\
 \maketitle
%\vspace{-15mm}
%\begin{center}{\it Dedicated to Vladimir Maz'ya on the occasion of his 70th
%birthday}
%\end{center}
\begin{abstract}
 The paper deals
with natural generalizations of the Hardy-Sobolev-Maz'ya
inequality and some related questions, such as the optimality and
stability of such inequalities, the existence of minimizers of the
associated variational problem, and the natural energy space
associated with the given functional.
\end{abstract}


\mysection{\bf Introduction}
%\def\theequation{1.\arabic{equation}}\makeatother
%\setcounter{equation}{0}
%%%%%%%%%%%%%%%%%%%%%%%%%%%%%%%%%%%%%%%%%%%%%%%%%%%%%%%%%%%%%%%%%%
The term ``inequalities of Hardy-Sobolev type" refers,
somewhat vaguely, to families of inequalities that in some
way interpolate the Hardy inequality
\begin{equation}
\label{Hardy} \int_{\Omega}|\nabla u(x)|^p\dx\ge C(N,p,K,\Omega)
\int_{\Omega} \dfrac{|u(x)|^p}{\mathrm{dist}(x,K)^p}\dx\qquad
u\in C_0^\infty(\Omega\setminus K),
\end{equation}
where $\Omega\subset\R^N$ is an open domain and
$K\subset\bar{\Omega}$ is a nonempty closed set, and the Sobolev
inequality
\begin{equation}
\int_{\Omega}|\nabla u(x)|^p\dx\ge C \left(\int_{\Omega}
|u(x)|^{p^*}\dx\right)^{p/p^*}\qquad u\in C_0^\infty(\Omega),
\end{equation}
where $C>0$, $1<p<N$, and $p^*\eqdef pN/(N-p)$ is the
corresponding Sobolev exponent. Throughout the paper we repeatedly
consider the following particular case.
 \begin{example}\label{ex1} {\em Let
$\Omega=\R^N=\R^n\times\R^m$, where $1\leq m < N$, and let
$K=\R^n\times \{0\}$. We denote the variables of $\R^n$ and $\R^m$
as $z$ and $y$ respectively, and set $\R_0^N\eqdef\R^n\times
(\R^m\setminus \{0\})$. It is well known that the Hardy inequality
\eqref{Hardy} holds with the best constant
\begin{equation}
C(N,p,\R^n\times \{0\},\R^N)=\left|\dfrac{m-p}{p}\right|^p.
\end{equation}

An elementary family of Hardy-Sobolev inequalities can be obtained
by H\"older interpolation between the Hardy and the Sobolev
inequalities. More significant inequalities of Hardy-Sobolev type
with the best constant in the Hardy term can be derived as
consequences of Caffarelli-Kohn-Nirenberg inequality (\cite{CKN,
Ilyin}) that provides estimates in terms of the weighted gradient
norm $\int |\xi|^\alpha|\nabla u|^p{\mathrm d}\xi$. The
substitution $u=|y|^\beta v$ into the Caffarelli-Kohn-Nirenberg
inequality can be used to produce inequalities that combine terms
with the critical exponent and with the Hardy potential. Such
inequalities are known as Hardy-Sobolev-Maz'ya (or HSM for
brevity) inequalities. In particular, in \cite[Section 2.1.6,
Corollary 3]{Mazya} Maz'ya proved the HSM inequality
\begin{multline}
\label{HSM}   \int_{\R_0^N}|\nabla
u|^2\dy\dz-\left(\dfrac{m-2}{2}\right)^2\int_{\R_0^N}
\dfrac{|u|^2}{|y|^2}\dy\dz\geq\\[3mm]
C\left(\int_{\R_0^N} |u|^{2^*}\dy\dz\right)^{2/2^*}\qquad
u\in C_0^\infty(\R_0^N),
\end{multline}
where $C>0$, $N>2$, and $1\le m<N$. This HSM inequality is false
for $m=N$ and reduces to the Sobolev inequality for $m=2$. Since
the left-hand side of \eqref{HSM} induces a Hilbert norm, the
inequality holds on $\mathcal D^{1,2}(\R_0^N)$, the completion of
$C_0^\infty(\R_0^N)$ in the gradient norm, which coincides with
$\mathcal D^{1,2}(\R^N)$ for all $m> 1$, in particular,
$C_0^\infty(\R_0^N)$ may be replaced by $C_0^\infty(\R^N)$ unless
$m=1$.
 }
\end{example}
A joint paper of Filippas, Maz'ya and Tertikas \cite{FMT} gives
the following generalization of the HSM inequality \eqref{HSM}.
 \begin{example}\label{ex2} {\em
Let $2\le p <N$, $p\neq m< N$, and let $\Omega\subset\R^N$ be a
bounded domain. Let $K$ be a compact $C^2$-manifold without
boundary embedded in $\R^N$, of codimension $m$ such that
$K\Subset\Omega$ for $1 < m <N$ (i.e., $K$ is compact in
$\Omega$), or $K =
\partial\Omega$ for $m = 1$. Assume further that
\be\label{convcond} -\Delta_p\left[
\,\mathrm{dist}\,(\cdot,K)^{(p-m)/(p-1)}\right]\ge 0 \qquad \mbox
{ in } \Omega\setminus K , \ee where
$\Delta_p(u)\eqdef\nabla\cdot(|\nabla u|^{p-2}\nabla u)$ is the
$p$-Laplacian. Then for all $u\in C_0^{\infty}(\Omega\setminus K)$
we have
\begin{equation}
\label{HSMp}  \int_\Omega|\nabla
u(x)|^p\dx-\left|\dfrac{m-p}{p}\right|^p\int_\Omega
\dfrac{|u(x)|^p}{\mathrm{dist}\,(x,K)^p}\dx \geq
C\left(\int_\Omega |u(x)|^{p^*}\dx\right)^{p/p^*} . \end{equation}
 }
\end{example}

For $N=3$ Benguria, Frank and Loss \cite{BFL} have shown recently
that the best constant $C$ in \eqref{HSM} is the Sobolev constant
$S_3$. Mancini and Sandeep \cite{Sandeep} have studied the analog
of HSM on the hyperbolic space and its close connection to the
original HSM inequality.

In the present paper we consider a nonnegative functional $Q$ of
the form
\begin{equation}
\label{Q} Q(u)\eqdef\int_\Omega\left(|\nabla u|^p+V|u|^p\right)\dx
\qquad u\in \core,
\end{equation}
where $\Omega\subseteq\R^N$ is a domain, $V\in
L_\mathrm{loc}^\infty(\Omega)$, and $1<p<\infty$. We study several
questions related to extensions of inequalities \eqref{HSM} and
\eqref{HSMp}. In Section~\ref{sec2}, we deal with generalizations
of these HSM inequalities for the functional $Q$. It turns out,
that in the subcritical case a {\em weighted} HSM inequality holds
true, where the weight appears in the Sobolev term. In the
critical case, one needs to add a Poincar\'e-type term (a
one-dimensional $p$-homogeneous functional), and we call it
Hardy-Sobolev-Maz'ya-Poincar\'e (or HSMP for brevity) inequality.
We show that under ``small" perturbations such HSM-type
inequalities are preserved (with the original Sobolev weight). We
also address the question concerning the optimal weight in the
generalized HSM inequality.

 In Section~\ref{sec4}, we study a natural energy space
$\mathcal D^{1,2}_V(\Omega)$ for nonnegative singular
Schr\"odinger operators, and discuss the existence of minimizers
for the HSM inequality in this space, that is, minimizers of the
equivalent Caffarelli-Kohn-Nirenberg inequality. Finally, in
Section~\ref{sec5} we prove that a related functional $\hat Q$
which satisfies  $C^{-1}Q\le\hat Q\le CQ$ for some $C>0$ induces a
norm on the cone of nonnegative $\core$-functions. For $p=2$, this
norm coincides (on the above cone) with the $\mathcal
D^{1,2}_V(\Omega)$-norm defined in \cite{ky2}. It is our hope that
this approach paves the way to circumvent the general lack of
convexity of the nonnegative functional $Q$ for $p\neq 2$.

\mysection{Generalization of HSM inequality}\label{sec2} We need
the following definition.
\begin{definition}
\label{GS}{\em Let $\Omega\subseteq\R^N$ be a domain, $V\in
L_\mathrm{loc}^\infty(\Omega)$, and $1<p<\infty$. Assume that the
functional
\begin{equation}
\label{Q1} Q(u)=\int_\Omega\left(|\nabla u|^p+V|u|^p\right)\dx
\end{equation}
is nonnegative on $\core$. A function  $\varphi\in C^1(\Omega)$ is
a {\em ground state} for the functional $Q$ if $\varphi$ is an
$L^p_{\mathrm{loc}}$-limit of a nonnegative sequence $\{
\varphi_k\}\subset \core$ satisfying
$$
 Q(\varphi_k)\to 0, \qquad \mbox{and }\quad
 \int_B|\varphi_k|^p\dx=1,
$$
for some fixed $B\Subset \Omega$ (such a sequence $\{\varphi_k\}$
is called a {\em null sequence}). The functional \eqref{Q} is
called {\em critical} if $Q$ admits a ground state and {\em
subcritical} or {\em weakly coercive} if it does not.
 }
\end{definition}
The following statement (see \cite{ky5}) is a generalization of
HSM inequality. Inequality \eqref{HSMP'} might be called
Hardy-Sobolev-Maz'ya-Poincar\'e (HSMP)-type inequality.
\begin{theorem}
\label{ky5} Let $Q$ be a nonnegative functional on $\core$ of the
form \eqref{Q}, and let $1<p<N$.

\vspace{3mm}

(i) The functional $Q$ does not admit a ground state if and only
if there exists a positive continuous function $W$ such that \be
\label{generalHSM}
  Q(u)\geq \left(\int_{\Omega} W|u|^{p^*}\mathrm{d}x\right)^{p/p^*} \qquad
 u\in C_0^\infty(\Omega).\ee

\vspace{3mm}

 (ii) If $Q$ admits a ground state $\varphi$, then $\varphi$ is the
unique global positive (super)-solution of the Euler-Lagrange
equation \be \label{groundstate}
 Q^\prime
(u)\eqdef -\Delta_p(u)+V|u|^{p-2}u=0\quad \mbox{in } \Omega.\ee
Moreover, there exists a positive continuous function $W$ such
that for every function $\psi\in C_0^\infty(\Omega)$ with
$\int_\Omega \psi \varphi \dx \neq 0$, the following inequality
holds
 \be\label{HSMP'}
  Q(u)+C\left|\int_{\Omega} \psi u\,\mathrm{d}x\right|^p
  \geq \left(\int_{\Omega} W|u|^{p^*}\mathrm{d}x\right)^{p/p^*}\qquad  u\in
C_0^\infty(\Omega)\ee with some suitable constant $C>0$.
 \end{theorem}
\begin{remark}{\em
For the relationships between the criticality of $Q$ in $\Omega$
and the $p$-capacity (with respect to the functional $Q$) of
closed balls see \cite[Theorem~4.5]{ky5} and \cite{Tr1,Tr2}. }
\end{remark}

Theorem~\ref{ky5} applies to the case of $\Omega=\R^N_0$ and the
Hardy potential (see Example~\ref{ex1}, and in particular
\eqref{HSM}), but it does not specify that the weight $W$ in the
Sobolev term is the constant function. We note that
Example~\ref{ex2} provides another Hardy-type functional
satisfying the HSM inequality with the weight
$W=\mathrm{constant}$.

On the other hand, let $\Omega=\R^N_0$, with  $m=N$, then the
corresponding Hardy functional admits a ground state
$\varphi(x)= |x|^{(p-N)/p}$, and therefore the HSM inequality
does not hold with any weight. Moreover, the HSMP inequality
\eqref{HSMP'}, which by Theorem~\ref{ky5} holds with some weight $W$,
is false with the weight $W=\mathrm{constant}$ (\cite{FT} and
Example~\ref{ex4}).

Let us present few other examples which illustrate further the
question of the admissible weights in the HSM and HSMP
inequalities. The first two examples are elementary but general.
In the first one the HSM inequality \eqref{generalHSM} holds with
the constant weight function, while in the second example
(Example~\ref{ex4}) such an inequality is false.
\begin{example}\label{ex3} {\em
Consider a nonnegative functional $Q$ of the form \eqref{Q}, where
$V\in L_\mathrm{loc}^\infty(\Omega)$ is nonzero function, and
$1<p<N$. For $\lambda\in \R$ we denote
$$
Q_\lambda(u)\eqdef \int_\Omega(|\nabla u|^p+\lambda V|u|^p)\dx.
$$
Then for every $\lambda\in(0,1)$ there exists $C>0$ such that
\be\label{HSM1} Q_\lambda(u)\ge C\|u\|^p_{p^*} \qquad
u\in\core,
 \ee
where $C=C(N,p,\lambda)>0$. This HSM inequality follows from
$$
Q_\lambda(u)=(1-\lambda)\int_\Omega|\nabla u|^p\dx+\lambda Q(u)\ge
(1-\lambda)\int_\Omega|\nabla u|^p\dx,
$$
and the Sobolev inequality.
 }
\end{example}

\begin{example}\label{ex4}{\em
Let $Q\ge 0$ be as in \eqref{Q}, where $1<p<N$.  Suppose that $Q$
admits ground state $\varphi\notin L^{p^*}(\Omega)$, and let
$\{\varphi_k\}$ be a null sequence (see Definition~\ref{GS}) such
that $\varphi_k\to\varphi$ locally uniformly in $\Omega$  (for the
existence of a locally uniform convergence null sequence, see
\cite[Theorem~4.2]{ky5}). Let $V_1\in L^\infty(\Omega)$ be a
nonzero nonnegative function with a compact support. Then
$$Q(\varphi_k)+\int_\Omega V_1|\varphi_k|^p\dx\to \int_\Omega V_1|\varphi|^p\dx<\infty,$$
while Fatou's lemma implies that $\|\varphi_k\|_{p^*}\to\infty$.
Therefore, the subcritical functional
$$
Q_{V_1}(u)\eqdef Q(u)+\int_\Omega V_1|u|^p\dx
$$
does not satisfy the HSM inequality \eqref{generalHSM} with the
constant weight. Similar argument shows that the critical
functional $Q$ does not satisfy the HSMP inequality with the
constant weight.
 }
\end{example}
\begin{remark}\label{rem1}{\em
Example~\ref{ex4} can be slightly generalized by replacing the
assumption $\varphi\notin L^{p^*}(\Omega)$ with  $\varphi\notin
L^{p^*}(\Omega, W\dx)$, where $W$ is a continuous positive weight
function. Under this assumption it follows that the functional
$Q_{V_1}$ and $Q$ do not satisfy HSM and respectively HSMP
inequality with the weight $W$.
 }
 \end{remark}
%%%%%%%%%%%%%%%%%%%%%%%%%%%%%%%%%%%%%%%%%%%%%%%%%%%%%
\begin{example}\label{ex6}{\em
In \cite[Theorem~C]{FTT}, Filippas, Tertikas and Tidblom proved
that a nonnegative functional $Q$ of the form \eqref{Q} with $p=2$
satisfies the HSM inequality in a smooth domain $\Omega$ with
$W=\mathrm{constant}$ if the equation $Q'(u)=0$ has a positive
$C^2$-solution $\varphi$ such that the following $L^1$-Hardy-type
inequality
$$
\int_{\Omega}\varphi^{2(N-1)/(N-2)}|\nabla u|\dx \ge
C\int_{\Omega}\varphi^{N/(N-2)}|\nabla \varphi|\,|u|\dx \qquad
 u\in C_0^\infty(\Omega).$$ holds true.
 }
\end{example}
%%%%%%%%%%%%%%%%%%%%%%%%%%%%%%%%%%%%%%%%%%%%%%%%%%
\begin{example}\label{ex5}{\em
Consider the function
$$X(r)\eqdef \left(|\log r|\right)^{-1} \qquad r>0.$$
Let $\Omega\subset\R^N$, $N>2$, be a bounded domain and let
$D>\sup_{x\in\Omega}|x|$.
The following inequality is due to Filippas and Tertikas 
\cite[Theorem A, and the corresponding Corrigendum]{FT}, see also \cite{AFT}.
\begin{multline}
\label{HSMFT}
\int_{\Omega}|\nabla
u|^2\dx-\left(\dfrac{N-2}{2}\right)^2\int_{\Omega}
\dfrac{|u|^2}{|x|^2}\dx\geq\\[3mm]
C\left(\int_{\Omega} |u|^{2^*}
X(|x|/D)^{1+N/(N-2)}\dx\right)^{2/2^*}\qquad  u\in
C_0^\infty(\Omega).
\end{multline}
%Moreover, the exponent $1+N/(N-2)$ in \eqref{HSMFT} cannot be
%decreased. In particular,
In this case the HSM inequality does not
hold with $W=\mathrm{constant}$ (cf. Example~\ref{ex4} and
Remark~\ref{rem1}).
 }
\end{example}
%%%%%%%%%%%%%%%%%%%%%%%%%%%%%%%
We now consider the question whether the weight $W$ in the HSM
inequality \eqref{generalHSM} is preserved (up to a constant
multiple) under small perturbations.
\begin{theorem}
\label{thm:VW} Let $\Omega$ be a domain in $\R^N$,  $N>2$,
and let $V\in L_\mathrm{loc}^\infty(\Omega)$.  Assume that the
following functional $Q$ satisfies the HSM inequality
\begin{equation}
\label{WW} Q(u)\eqdef\int_\Omega\left(|\nabla
u|^2+V|u|^2\right)\dx \ge\left(\int_\Omega
W|u|^{2^*}\dx\right)^{2/2^*} \qquad u\in \core
\end{equation}
 with some
positive continuous function $W$. Let $\tilde{V}\in
L^{\infty}_{\mathrm{loc}}(\Omega)$ be a nonzero potential
satisfying \be \label{VW} |\tilde{V}|^{N/2}W^{(2-N)/2}\in
L^{1}(\Omega), \ee and consider the one-parameter family of
functionals $\tilde{Q}_\lambda$ defined by
$$
\tilde{Q}_\lambda(u)\eqdef Q(u)+\lambda\int_\Omega
\tilde{V}|u|^2\dx,
$$
where $\lambda \in\R$.

(i) If $\tilde{Q}_\lambda$ is nonnegative on $\core$ and does not
admit a ground state, then \be \label{subcrt} \tilde{Q}_\lambda
(u)\geq C\left(\int_\Omega W|u|^{2^*}\dx\right)^{2/2^*} \qquad
u\in \core, \ee where $C$ is a positive constant.

(ii) If $\tilde{Q}_\lambda$ is nonnegative on $\core$ and admits a
ground state $v$, then for every $\psi\in\core$ such that
$\int_\Omega\psi v\dx\neq 0$ we have  \be \label{crtcal}
\tilde{Q}_\lambda(u)+C_1\left(\int_\Omega\psi u\dx\right)^2 \ge
C\left(\int_\Omega W|u|^{2^*}\dx\right)^{2/2^*} \qquad u\in \core
\ee with suitable positive constants $C, C_1>0$.

(iii) The set $$ S\eqdef \{\lambda\in\R\mid \tilde{Q}_\lambda\geq
0
 \mbox{ on } \core\}$$ is a closed interval with a nonempty
interior which is bounded if and only if $\tilde{V}$ changes its
sign on a set of a positive measure in $\Omega$. Moreover,
$\lambda \in
\partial S $ if and only if $\tilde{Q}_\lambda$ is critical in $\Omega$.
\end{theorem}
\begin{proof} (i)--(ii)
 Let $\mathcal D^{1,2}_{\lambda\tilde{V}}(\Omega)$ denote the completion of $\core$ with
respect to the norm defined by the square root of the left-hand
side of \eqref{subcrt} if $\tilde{Q_\lambda}$ does not admit a
ground state, and by the square root of the left-hand side of
\eqref{crtcal} if $\tilde{Q_\lambda}$ admits a ground state (see
\cite{ky2}). Similarly, we denote by $\mathcal
D^{1,2}_{V}(\Omega)$ the completion of $\core$ with respect to the
norm defined by the square root of the left-hand side of
\eqref{WW}. We denote the norms on $\mathcal
D^{1,2}_{\lambda\tilde{V}}(\Omega)$ and $\mathcal
D^{1,2}_{V}(\Omega)$ by $\|\cdot\|_{\mathcal
D^{1,2}_{\lambda\tilde{V}}}$ and $\|\cdot\|_{\mathcal
D^{1,2}_{V}}$ respectively.

Assume  that \eqref{subcrt} (respect. \eqref{crtcal}) does not
hold. Then there exists a sequence $\{u_k\}\subset\core$ such that
 \be\label{null}
\|u_k\|_{\mathcal D^{1,2}_{\lambda\tilde{V}}}\to 0,  \qquad
\mbox{and } \int_\Omega W|u_k|^{2^*}\dx=1.
 \ee
  By \cite[Proposition~3.1]{ky2},
the space $\mathcal D^{1,2}_{\lambda\tilde{V}}(\Omega)$ is
continuously imbedded into  $W^{1,2}_{\mathrm{loc}}(\Omega)$ and
therefore, $u_k\to 0$ in $W^{1,2}_{\mathrm{loc}}(\Omega)$.
Consequently, for any $K\Subset \Omega$ we have \be\label{intK}
\lim_{k\to\infty}\int_{K} |\tilde{V}||u_k|^2\dx=0. \ee

On the other hand, \eqref{VW} and  H\"older inequality imply that
for any $\varepsilon>0$ there exists $K_\varepsilon\Subset \Omega$
such that
 \be\label{holder}
\left|\int_{\Omega\setminus K_\varepsilon}\!\!\!\!
\tilde{V}|u_k|^2\dx\right|\!\leq \!\left(\int_{\Omega\setminus
K_\varepsilon}\!\!\!|\tilde{V}|^{N/2}W^{(2-N)/2}\dx\right)^{2/N}
\!\!\left(\int_\Omega
\!\!W|u_k|^{2^*}\dx\right)^{2/{2}^*}\!\!\!\!<\varepsilon.
 \ee

 Since
$$\|u_k\|_{\mathcal D^{1,2}_{V}}\leq \|u_k\|_{\mathcal
D^{1,2}_{\lambda\tilde{V}}}+\left|\int_\Omega
\lambda\tilde{V}|u_k|^2\dx\right|^{1/2},$$
it follows from \eqref{null}--\eqref{holder} that the sequence
$u_k \to 0$  in $\mathcal D^{1,2}_V(\Omega)$. Therefore,
\eqref{WW} implies that  $\int_\Omega W|u_k|^{2^*}\dx\to 0$ which
contradicts the assumption $\int_\Omega W|u_k|^{2^*}\dx=1$.
Consequently, \eqref{subcrt} (resp. \eqref{crtcal}) holds true.

\vskip 3mm

(iii) It follows from \cite[Proposition~4.3]{ky3} that $S$ is an
interval, and that $\lambda \in \mathrm{int}\,S$ implies that
$Q_\lambda$ is subcritical in $\Omega$. The claim on the
boundedness of $S$ is trivial and left to the reader.

On the other hand,  suppose that for some $\lambda \in\R$ the
functional  $\tilde{Q}_\lambda$ is subcritical. By part (i),
$\tilde{Q}_\lambda$ satisfies the HSM inequality with weight $W$.
Therefore, \eqref{holder} (with $K_\varepsilon=\emptyset$) implies
that
\begin{equation}\label{0}
 \tilde{Q}_\lambda(u)\geq C\left(\int_\Omega
W|u|^{2^*}\dx\right)^{2/2^*}\geq C_1 \left|\int_\Omega
\tilde{V}|u|^{2}\dx\right| \qquad u\in \core.
\end{equation}
Therefore, $\lambda\in \mathrm{int}\,S$.  Consequently,
$\lambda\in
\partial S$ implies that  $\tilde{Q}_{\lambda}$ is critical in $\Omega$. In particular, $0\in
\mathrm{int}\,S $.
\end{proof}
\begin{example}\label{ex7}
{\em Let $\Omega=\R^N$, where $N\geq 3$, and let $V\in
L^{N/2}(\R^N)$ such that $V\ngeq 0$ (so, $V$ is a short range
potential). Fix $\mu<(N-2)^2/4$. Then the classical Hardy
inequality together with Example~\ref{ex3} and
Theorem~\ref{thm:VW} imply that there exists $\lambda^*>0$ such
that for $\lambda<\lambda^*$, we have the following HSM inequality
\begin{multline}
\label{HSMmu}   \int_{\R^N}|\nabla u|^2\dx-\mu\int_{\R^N}
\dfrac{|u|^2}{|x|^2}\dx+\lambda \int_{\R^N}
V(x)|u|^2\dx\geq\\[3mm]
C_\lambda\left(\int_{\R^N} |u|^{2^*}\dx\right)^{2/2^*}\qquad
 u\in C_0^\infty(\R^N).
\end{multline}
On the other hand, if $\lambda=\lambda^*$, then the associated
functional is critical and satisfies the corresponding HSMP
inequality with the weight function $W=\mathrm{constant}$. Recall
that the HSM and HSMP inequalities for $\mu=(N-2)^2/4$ are false
with the weight $W=\mathrm{constant}$ (see, Example~\ref{ex4} and
\cite{FT}).
 }
\end{example}
\begin{example}\label{ex8} {\em Consider again Example~\ref{ex2} with
$p=2<N$, and $2\neq m<N$. By \cite[Theorem~1.1]{FMT}, there exists
$M\leq 0$ such that
 the following HSM
inequality holds true \begin{multline}\label{HSM2}
Q(u)\eqdef\int_\Omega|\nabla
u|^2\dx-\left(\frac{m-2}{2}\right)^2\int_\Omega\!\dfrac{|u|^2}{\mathrm{dist}\,(x,
K)^2}\dx - M \int_\Omega|u|^2\dx\\[3mm] \geq
C\left(\int_\Omega|u|^{2^*}\dx\right)^{2/2^*}\quad u\in
C^\infty_0(\Omega\setminus K).
\end{multline}
We note that if \eqref{convcond} is satisfied, then \eqref{HSM2}
holds with $M=0$.

Let  $V\in  L^{\infty}_{\mathrm{loc}}(\Omega)\cap L^{N/2}(\Omega)$
be a nonzero function, and consider the one-parameter family of
functionals $Q_\lambda$ defined by
$$Q_\lambda(u)\eqdef Q(u)+\lambda\int_\Omega V|u|^2\dx, $$
where $\lambda \in\R$. By Theorem~\ref{thm:VW}, the set $S$ of all
$\lambda$ such that $Q_\lambda$ is nonnegative on $\core$ is a
nonempty closed interval with a nonempty interior. Moreover, for
$\lambda \in \mathrm{int}\,S$ there exists  a positive constant
$c_{\lambda}$ such that
 \be\label{HSM6} Q_\lambda(u) \ge
c_{\lambda}\left(\int_\Omega|u|^{2^*}\dx\right)^{2/2^*}\qquad
 u\in C^\infty_0(\Omega\setminus K).
 \ee
On the other hand, if $\lambda \in \partial S$, then $Q_\lambda$
admits a ground state $v$. Therefore, Theorem~\ref{thm:VW} implies
that for every $\psi\in C^\infty_0(\Omega\setminus K)$ satisfying
$\int_\Omega\psi v\dx\neq 0$ there exist constants $C, C_1>0$ such
that
$$
Q_{\lambda}(u) +C\left(\int_\Omega u\psi\dx\right)^2 \geq
c_1\left(\int_\Omega|u|^{2^*}\dx\right)^{2/2^*}\quad  u\in
C^\infty_0(\Omega\setminus K).
$$
We note that if $K=\partial \Omega$ is smooth (that is, $m = 1$)
and $V=\mathbf{1}$, one actually deals with the case considered by
Brezis and Marcus in \cite[Theorem~1.1]{BM}.  In particular, let
$\lambda^*$ be the supremum of all  $\lambda\in \R$ such that the
inequality \be \label{BM0}
 \int_\Omega|\nabla
u|^2\dx-\frac{1}{4}\int_\Omega\!\dfrac{|u|^2}{\mathrm{dist}\,(x,
\partial \Omega)^2}\dx - \lambda \int_\Omega|u|^2\dx \geq 0 \qquad
u\in\core \ee holds true ($\lambda^*>-\infty$ and is attained by
\cite[Theorem~1.1]{BM}). Then Theorem~\ref{thm:VW} implies that
for each $\lambda<\lambda^*$ there exists $C_\lambda>0$ such that
\begin{multline} \label{BM}
 \int_\Omega|\nabla
u|^2\dx-\frac{1}{4}\int_\Omega\!\dfrac{|u|^2}{\mathrm{dist}\,(x,
\partial \Omega)^2}\dx - \lambda \int_\Omega|u|^2\dx\\[3mm] \geq
C_\lambda\left(\int_\Omega|u|^{2^*}\dx\right)^{2/2^*}\quad
u\in\core.
\end{multline}
Moreover, Theorem~\ref{thm:VW} implies that for $\lambda
=\lambda^*$, the functional defined by the left-hand side of
\eqref{BM} is critical, and satisfies the HSMP inequality with
weight $W=\mathrm{constant}$. In particular, the corresponding
Euler-Lagrange equation $Q_{\lambda^*}'(u)=0$ in $\Omega$  admits
a unique positive (super)-solution.

Theorem~1.1 of \cite{BM} has been extended by Marcus and Shafrir
in \cite[Theorem~1.2]{MS} to the case $1<p<\infty$ and a
perturbation $0< V(x)=O(\mathrm{dist}\,(x,
\partial \Omega)^{\gamma})$, where $\gamma > -p$
(cf. our assumption \eqref{VW}, where $p=2$). Following \cite{MS},
let $\lambda^*$ be the supremum of all $\lambda\in \R$ such that
the inequality \be \label{BM1}
 \int_\Omega|\nabla
u|^2\dx-\frac{1}{4}\int_\Omega\!\dfrac{|u|^2}{\mathrm{dist}\,(x,
\partial \Omega)^2}\dx - \lambda \int_\Omega V(x)|u|^2\dx \geq 0 \quad
u\in\core. \ee holds true. It follows that Theorem~\ref{thm:VW}
with the constant weight applies also to this functional if in
addition $V\in L^{\infty}_{\mathrm{loc}}(\Omega)\cap
L^{N/2}(\Omega)$.
 }
\end{example}
\begin{remark}\label{rem2}{\em
We note that even under the less restricted assumptions of
\cite[Theorem~1.2]{MS}, with $p=2$ and $\lambda=\lambda^*$, one
can show that the positive solution $u_*$ of Equation (1.14) in
\cite{MS} is actually a ground state. Therefore, $u_*$ is the
unique (up to a multiplicative constant) global positive
supersolution of that equation, and the corresponding functional
is critical.

Indeed, Lemma~5.1 of \cite{MS} implies that any positive
supersolution of \cite[Equation~(1.14)]{MS} satisfies
\be\label{in5}Cu(x)\geq \mathrm{dist}\,(x,
\partial \Omega)^{1/2}\qquad x\in\Omega.\ee
On the other hand, \cite[Theorem~1.2]{MS} implies that the
positive solution $u_*$ satisfies
 \be\label{asy5}
u_*(x) \asymp \mathrm{dist}\,(x,
\partial \Omega)^{1/2} \qquad x\in\Omega,\ee
where $f\asymp g$ means that there exists a positive constant $C$
such that $C^{-1}\leq f/g\leq C$ in $\Omega$. Now, take a positive
supersolution $u$, and let $\varepsilon$ be the maximal positive
number such that  $u-\varepsilon u_*\geq 0$ in $\Omega$. Note that
by \eqref{in5} and \eqref{asy5}, $\varepsilon$ is well defined. By
the strong maximum principle it follows that either $u=\varepsilon
u_*$, or $u-\varepsilon u_*>0$. Consequently, \eqref{in5} and
\eqref{asy5} imply that there exists a positive constant $C_1$
such that
$$u-\varepsilon u_*\geq C\mathrm{dist}\,(x,
\partial \Omega)^{1/2}\geq C_1u_*\qquad \mbox{ in } \Omega,$$
which is a contradiction to the definition of $\varepsilon$.
 }
\end{remark}

\mysection{The space $\mathcal D^{1,2}_V(\Omega)$ and minimizers
for the HSM inequality}\label{sec4}
 Consider again the HSM inequality
\eqref{HSM}. This inequality clearly extends to $\mathcal
D^{1,2}(\R^N)$ for $m>2$ and to $\mathcal D^{1,2}(\R^N_0)$ for
$m=1$, but since the quadratic form $Q(u)$ in the left-hand side
of \eqref{HSM} induces a scalar product on $C_0^\infty(\R^N_0)$,
the natural domain of $Q$ is the completion of
$C_0^\infty(\R^N_0)$ with respect to the norm $Q(\cdot)^{1/2}$.
Recall \cite{ky2} that given a general subcritical functional $Q$
of the form \eqref{Q} (with $p=2$), we denote such a completion by
$\mathcal D^{1,2}_V(\Omega)$. Similarly to the standard definition
of $\mathcal D^{1,2}(\R^N)$ for $N=1,2$, when $Q$ admits a ground
state, one appends to $Q(u)$ a correction term of the form
$\left(\int_{\Omega} \psi u \dx\right)^2$. Hence,  by
\eqref{generalHSM} and \eqref{HSMP'} the space $\mathcal
D^{1,2}_V(\Omega)$ is continuously imbedded into a weighted
$L^{2^*}$-space.

In the particular case \eqref{HSM}, $V$ is the Hardy potential
$[(m-2)/2]^2|y|^{-2}$. By \eqref{HSM}, the space $\mathcal
D^{1,2}_V(\R^N_0)$ is continuously imbedded into
$L^{2^*}(\R^N_0)$, thus its elements can be identified as
measurable functions. The substitution $u=|y|^{(2-m)/2}v$
transforms HSM inequality \eqref{HSM} into an inequality of
Caffarelli-Kohn-Nirenberg type:
\begin{multline} \label{HSM'} \int_{\R^N}|y|^{2-m}|\nabla v|^2\dy\dz\\[3mm] \geq
C\left(\int_{\R^N} |y|^{(2-m)2^*/2}|v|^{2^*}\dy\dz\right)^{2/2^*}
\;\;\;\;  v\in \mathcal
D^{1,2}(\R_0^N,|y|^{2-m}\dy\dz).\end{multline}

The left-hand side of \eqref{HSM'} defines a Hilbert space
isometric to $\mathcal D^{1,2}_V(\R_0^N)$. However, the Lagrange
density \be \label{Lagr} |\nabla
u|^2-\left(\frac{m-2}{2}\right)^2\frac{|u|^2}{|y|^2}\ee is no
longer integrable for an arbitrary $u\in\mathcal
D^{1,2}_V(\R^N_0)$. The integrable Lagrange density of
\eqref{HSM'}, $|y|^{2-m}|\nabla (u|y|^{(m-2)/2})|^2$ can be
equated to \eqref{Lagr} by partial integration when $u\in
C_0^\infty(\R^N_0)$, but this connection does not extend to the
whole of $\mathcal D^{1,2}_V(\R^N_0)$ as the terms that mutually
cancel in the partial integration on $C_0^\infty(\R^N_0)$ might
become infinite. In particular, it should not be expected a priori
that the minimizer for HSM inequality in $\mathcal
D^{1,2}_V(\R^N_0)$ would have a finite gradient in
$L^2(\R^N_0,\dx)$.

Existence of minimizers for the variational problem associated
with \eqref{HSM'} is proved in \cite{TerTin} for all codimensions
$0<m<N$, where $N>3$. The existence proof is based on
concentration compactness argument that utilizes invariance
properties of the problem. Similarly to other problems where lack
of compactness stems from a noncompact equivariant group of
transformations, some general domains and potentials admit
minimizers and some do not, and analogy with similar elliptic
problems in $\mathcal D^{1,2}(\R^N)$ provides useful insights (see
for example \cite{Smets}).
\mysection{Convexity properties of $Q$ for $p>2$}\label{sec5} The
definition of $\mathcal D^{1,2}_V(\Omega)$ cannot be applied to
other values of $p$, since for $p\neq 2$ the positivity of the
functional $Q$ on $\core$ does not necessarily imply its
convexity, and thus it does not give rise to a norm.  For the lack
of convexity when $p>2$, see an elementary one-dimensional
counterexample at the end of \cite{dPEM}, and also the proof of
Theorem~7 in \cite{GS}. For $p<2$, see \cite[Example~2]{FHTdT}.

On the other hand, by \cite[Theorem~2.3]{ky3}, the functional  $Q$
is nonnegative on $\core$ if and only if the equation $Q'(u)=0$ in
$\Omega$ admits a positive global solution $v$. With the help of
such a solution $v$, one has the identity \cite{DS,AH1,AH2}:
$$
Q(u)=\int_\Omega L_v(w)\dx \qquad u\in\coreplus,
$$
where $w\eqdef {u}/{v}$, the Lagrangian $L_v(w)$ is defined by
 \be\label{PiconeLagr}
L_v(w)\eqdef |v\nabla w+w\nabla v |^p-w^p|\nabla v|^p
-pw^{p-1}v|\nabla v|^{p-2}\nabla v\cdot\nabla w\geq 0 \quad
w\in\coreplus,
 \ee
and $\coreplus$ denotes the cone of all nonnegative functions in
$\core$.

The following proposition claims that the nonnegative Lagrangian
$L_v(w)$, which contains indefinite terms,  is bounded from above
and from below by multiples of a simpler Lagrangian.
\begin{proposition}[{\cite[Lemma~2.2]{aky}}]
\label{prop:superPicone}  Let $v$ be a positive solution of the
equation $Q'(u)=0$ in $\Omega$. Then
\bea \label{p<2}  L_v(w)\asymp  v^2 |\nabla w|^2\left(w|\nabla
v|+v|\nabla w|\right)^{p-2}\qquad \forall w\in\coreplus.\eea
In particular, for $p\ge 2$, we have
\bea \label{p>2}  L_v(w)\asymp  \hat L_v(w)\eqdef v^p|\nabla
w|^p+v^2|\nabla v|^{p-2} w^{p-2}|\nabla w|^2 \qquad \forall
w\in\coreplus.\eea
\end{proposition}
\vskip 4mm
 Define the simplified energy $\hat{Q}$ by
\begin{equation}\label{Qhat}
\hat{Q}(u)\eqdef \int_\Omega \hat L_v(w)\dx \qquad w =u/v
\in\coreplus.
\end{equation}
It is shown in \cite{aky} that for $p>2$ neither of the terms in
the simplified energy $\hat{Q}$ is dominated by the other.

 It follows
from Proposition~\ref{prop:superPicone} that
$$Q(u)=Q(|u|)\asymp\hat{Q}(|u|) \qquad u
\in\core.$$

In \cite{TakTin}, the solvability of equation $Q'(u)=f$ is proved
in the class of functions $u$ satisfying $Q^{**}(u)<\infty$, where
$Q^{**}\le Q$ is the second convex conjugate (in the sense of
Legendre transformation) of $Q$. If the inequality $Q\le C Q^{**}$
is true, then ${Q^{**}}^{1/p}(u)$ would define a norm, and $Q$
would extend to a Banach space, which should be regarded as the
natural energy space for the functional $Q$.

On the other hand, if $p>2$, it is not clear whether the functional
$\hat{Q}$ is convex due to the second term in \eqref{p>2}.
It has, however, the following convexity property.
\begin{proposition}
\label{convexity} Assume that $p\geq 2$, and let $v\in
C^1_{\mathrm{loc}}(\Omega)$ be a fixed positive function. Consider
the functional
$$
\mathcal Q(\psi)\eqdef \hat{Q}(v\psi^{2/p})\qquad
\psi\in\coreplus,
$$
where $\hat Q$ is defined by \eqref{p>2} and
\eqref{Qhat}. Then the functional $\mathcal Q$
is convex on $\coreplus$.
\end{proposition}
\begin{proof} We first split each of the functionals $\hat{Q}$ and $\mathcal{Q}$ into a sum of two
functionals:
$$
\hat{Q}_1(u)\!\eqdef\!\displaystyle{\!\int_\Omega\!\!\! v^p|\nabla
w|^p\!\dx}, \;\; \hat{Q}_2(u)
\!\eqdef\!\displaystyle{\!\int_\Omega\!\!\! v^2|\nabla
v|^{p-2}w^{p-2}|\nabla w|^2\!\dx}
 \;\;\;  w\! =\!u/v\!\in\! \coreplus,$$

\begin{align*}\label{lc}
&\mathcal Q_1(\psi)\eqdef
\hat{Q}_1(v\psi^{2/p})=\displaystyle{\int_\Omega\ v^p|\nabla
(\psi^{2/p})|^p\dx} & \qquad \psi\in \coreplus ,\\[4mm]
 & \mathcal Q_2(\psi)\eqdef \hat{Q}_2(v\psi^{2/p})=
 \displaystyle{\int_\Omega
v^2|\nabla v|^{p-2}\psi^{2(p-2)/p}|\nabla (\psi^{2/p})|^2\dx}
 &   \qquad \psi\in \coreplus.
\end{align*}
Thus,  $\hat{Q}=\hat{Q}_1+\hat{Q}_2$, and
$\mathcal{Q}=\mathcal{Q}_1+\mathcal{Q}_2$.

\vskip 4mm

For $t\in[0,1]$ and $w_0,w_1\in \coreplus$, let
$$
w_t\eqdef\left[(1-t)w_0^{p/2}+tw_1^{p/2}\right]^{2/p}.
$$
Then
$$
\nabla w_t=\dfrac{(1-t)w_0^{p/2-1}\nabla w_0+tw_1^{p/2-1}\nabla
w_1} {\left[(1-t)w_0^{p/2}+tw_1^{p/2}\right]^{1-2/p}}\,.
$$
Therefore,
 \be\label{eq9} |\nabla
w_t|\le\dfrac{[(1-t)^{2/p}w_0]^{p/2-1}(1-t)^{2/p} |\nabla
w_0|+(t^{2/p}w_1)^{p/2-1}t^{2/p}|\nabla w_1|}
{\left[(1-t)w_0^{p/2}+tw_1^{p/2}\right]^{1-2/p}}\,.
 \ee
Applying H\"older inequality to the sum in the numerator of
\eqref{eq9} (with the terms $(1-t)^{2/p}|\nabla w_0|$ and
$t^{2/p}|\nabla w_1|$ raised to the power $p/2$) and taking into
account that the conjugate of $p/2$ is reciprocal to $1-2/p$, we
have
\begin{equation}
\label{p/2}
|\nabla w_t|^{p/2}\le (1-t)|\nabla w_0|^{p/2}+t|\nabla w_1|^{p/2}.
\end{equation}
From \eqref{p/2} it follows easily that
$$
|\nabla w_t|^{p}\le (1-t)|\nabla w_0|^{p}+t|\nabla w_1|^{p}.
$$
Setting $\psi_t\eqdef w_t^{p/2}$, $t\in[0,1]$, we immediately
conclude that $\mathcal Q_1$ is convex as a function of $\psi$.
The same conclusion extends to $\mathcal Q_2$ once we note that
$$
w^{p-2}|\nabla w|^2=(2/p)^2|\nabla w^{p/2}|^2 ,
$$
and use \eqref{p/2} for $p=4$.
\end{proof}
\vskip4mm
Let
\begin{equation}
\label{N} N(\psi)\eqdef \left[
\mathcal{Q}(\psi)\right]^{1/2}=\left[\hat
Q(v\psi^{2/p})\right]^{1/2}\qquad \psi\in\coreplus.
\end{equation}
It is immediate that $N(\psi)>0$ for $\psi\in\coreplus$,  unless
$\psi=0$, and that $N(\lambda\psi)=\lambda N(\psi)$ for $\lambda
\geq 0$. Due to Proposition~\ref{convexity}, the functional
$N(\cdot)$ satisfies the triangle inequality
$$N(\psi_1+\psi_2)\leq N(\psi_1)+N(\psi_2)\qquad \psi_1,\psi_2\in\coreplus.$$
Thus, we have equipped the cone $\coreplus$ with a norm. For $p=2$
the functional $Q=\hat Q$ is a positive quadratic form, and thus
convex. Consequently, in the subcritical case, $Q^{1/2}$ extends
the functional $N$ to a norm on the whole $\core$, and then by
completion,  to the Hilbert space $\mathcal D^{1,2}_V(\Omega)$. It
would be interesting to introduce $\mathcal D^{1,p}_V(\Omega)$ for
$p>2$ once one finds an extension of $N$ to $\core$.

%%%%%%%%%%%%%%%%%%%%%%%%%%%%%%%%%%%%%%%%%%%%%%%%%%%%%%%%%%%
%

\begin{center}{\bf Acknowledgments} \end{center}
Part of this research was done while K.~T. was visiting the
Technion. K.~T. would like to thank the Technion for the kind
hospitality. Y.~P. acknowledges the support of the Israel Science
Foundation (grant No. 587/07) founded by the Israeli Academy of
Sciences and Humanities, and the B.~and G.~Greenberg Research Fund
(Ottawa).
%
%%%%%%%%%%%%%%%%%%%%%%%%%%%%%%%%%%%%%%%%%%%%%%%%%%%%%

\end{document}